\documentclass[11pt]{article}
\usepackage{amsmath,amssymb,amsthm,graphics,latexsym,amsfonts}
\usepackage{fancyhdr}
\usepackage{color}
\usepackage{graphics}

\newtheorem{theorem}{Theorem}
\newtheorem{lemma}[theorem]{Lemma}

\newtheorem{problem}[theorem]{Problem}

\usepackage{latexsym,amsmath,amssymb,amsfonts,epsfig,graphicx,cite,psfrag}
\usepackage{eepic,color,colordvi,amscd}
\usepackage{ebezier}
\usepackage{verbatim}

\usepackage{subfigure}

\usepackage{indentfirst}

\begin{document}

\title{\LARGE Regular graphs with equal matching number and independence number\thanks{Supported by the National Natural
Science Foundation of China under grant No.11471257 and Fundamental Research Funds for the Central Universities}}

\author{ {Hongliang Lu \footnote{Corresponding email: luhongliang215@sina.com (H. Lu)}, Zixuan Yang}\\
\small School of Mathematics and Statistics, Xi¡¯an Jiaotong University \\ \small  Xi'an, Shaanxi 710049, P.R.China \\ }

\date{}

\maketitle {\small {\bfseries \centerline{Abstract}}

\vspace{3ex}

Let $r\geq 3$ be an integer and $G$ be a graph. Let $\delta(G), \Delta(G)$, $\alpha(G)$ and $\mu(G)$ denotes minimum degree, maximum degree, independence number and matching number of $G$, respectively. Recently, Caro, Davila and Pepper proved $\delta(G)\alpha(G)\leq \Delta(G)\mu(G)$. Mohr and Rautenbach characterized the extremal graphs for non-regular graphs and 3-regular graphs. In this note, we characterize the extremal graphs for all $r$-regular graphs in term of Gallai-Edmonds Structure Theorem, which extends Mohr and Rautenbach's result.

\vspace{3ex}


\section {\large Introduction}

\noindent In this paper, we consider finite undirected graphs without loops. Let $G$ be a graph with vertex set $V(G)$ and edge set $E(G)$. The number of vertices of $G$ is called its \emph{order} and denoted by $|V(G)|$. On the other hand, the number of edges in $G$ is called its \emph{size} and denoted by $e(G)$. For a vertex $u$ of a graph $G$, the \emph{degree} of $u$ in $G$ is denoted by $d_G(u)$, and the minimum and maximum vertex degrees of $G$ will be denoted $\delta(G)$ and $\Delta(G)$, respectively.
The set of vertices adjacent to $u$ in $G$ is denoted by $N_{G}(u)$. For $S\subseteq V(G)$, the subgraph of $G$ induced by $S$ is denoted by $G[S]$. For two disjoint subsets $S,T\subseteq V(G)$, let $E_{G}(S,T)$ denote the set of edges of $G$ joining $S$ to $T$ and let $e_{G}(S,T)=|E_{G}(S,T)|$. A component is \emph{trivial} if it has no edges; otherwise it is \emph{nontrivial}.

A \emph{matching} of a graph is a set of edges such that no two edges share a vertex in common. For a matching $M$, a vertex $u$ of $G$ is called \emph{saturated} by $M$ if $u$ is incident to an edge of $M$. A matching $M$ is a \emph{maximum matching} of $G$ if there does not exist a matching $M'$ of $G$ such that $|M'|>|M|$. A \emph{perfect matching} of a graph is a matching saturating all vertices. The cardinality of a maximum matching is called the \emph{matching number} of $G$ and is denoted by $\mu(G)$.
An \emph{independent set} is a set of vertices in a graph, no two of which are adjacent. A \emph{maximum independent set} is an independent set of largest possible size for a given graph $G$. The cardinality of a maximum independent set is called the \emph{independence number} of $G$ and is denoted by $\alpha(G)$.

There are many relationships between the graph parameters $\alpha(G)$ and $\mu(G)$. It is known that $\lfloor\frac{n}{2}\rfloor+1\leq \alpha(G)+\mu(G)\leq |V(G)|\leq \alpha(G)+2\mu(G)$ holds for every graph $G$. If $\alpha(G)+\mu(G)=|V(G)|$, then $G$ is called \emph{K{\"o}nig-Egerv\'{a}ry graph} \cite{Dem,Ste}. It is easy to see that if $G$ is a K\H{o}nig-Egerv\'{a}ry graph, then $\alpha(G)\geq\mu(G)$. The K{\"o}nig-Egerv\'{a}ry graph have been extensively studied in \cite{Bon,Bou,Lev,Lov}.

Recently, Levit et al.\cite{Lev2} showed that $\alpha(G)\leq\mu(G)$ under the condition that $G$ contains an unique odd cycle. Caro, Davial and Pepper \cite{Car} obtained the following results.


\begin{theorem}[Caro, Davial and Pepper,  \cite{Car}] If $G$ is a graph, then
$$\delta(G)\alpha(G)\leq \Delta(G)\mu(G),$$
and this bound is sharp.
\end{theorem}

\begin{theorem}[Caro, Davial and Pepper, \cite{Car} ]\label{I-M-bound} If $G$ is a r-regular graph with $r>0$, then
$$\alpha(G)\leq \mu(G).$$
\end{theorem}

They also proposed the following two open problems.

\begin{problem}[Caro, Davial and Pepper, \cite{Car}]\label{prob1}
Characterize $\alpha(G)=\mu(G)$ whenever $G$ is 3-regular.
\end{problem}

\begin{problem}[Caro, Davial and Pepper, \cite{Car}]\label{prob2}
Characterize all graphs $G$ for which
$\delta(G)\alpha(G)=\Delta(G)\mu(G)$.
\end{problem}

Mohr and Rautenbach \cite{Moh}  characterized the non-regular extremal graphs as well as 3-regular graphs, which solved Problems \ref{prob1} and \ref{prob2}. In the  note, we characterize $r$-regular graphs $G$ with $\alpha(G)=\mu(G)$ in term of Gallai-Edmonds Structure Theorem.

Now we firstly introduce Gallai-Edmonds Structure Theorem \cite{Yu}. For a graph $G$, denote by $D_G$ the set of all vertices in $G$ which are not saturated by at least one maximum matching of $G$. Let $A_G$ be the neighbor set of $D_G$, i.e., the set of vertices in $V(G)-D_G$ adjacent to at least one vertex in $D_G$. Finally let $C_G= V(G)-D_G-A_G$. Clearly, this partition is well-defined for every graph and dose not rely on the choices of maximum matchings. A graph $G$ is said to be \emph{factor-critical} if $G-v$ has a perfect matching for any vertex $v\in V(G)$. A matching is said to be a \emph{near-perfect matching} if it covers all vertices but one. For a bipartite graph $H=(A,B)$, the set $A$ with \emph{positive surplus} if $|N_{H}(X)|>|X|$ for every non-empty subset $X$ of $A$. The subgraph of $G$ induced by a vertex subset $S$ is denoted by $G[S]$.

\begin{theorem}[Gallai-Edmonds Structure Theorem, see \cite{Yu}]\label{GE-thm}Let $G$
be a graph and let $D_G$, $C_G$ and $A_G$ be the vertex-partition defined above. Then
\begin{itemize}
\item [\rm{(i)}]the component of the subgraph induced by $D_G$ are factor-critical;
\item [\rm{(ii)}]the subgraph induced by $C_G$ has a perfect matching;
\item [\rm{(iii)}]if $M$ is any maximum matching of $G$, it contains a near-perfect matching of each component of $D_G$, a perfect matching of each component of $C_G$ and matches all vertices of $A_G$ with vertices in distinct component of $D_G$;
\item [\rm{(iv)}]the bipartite graph obtained from $G$ by deleting the vertices of $C_G$ and the edges spanned by $A_G$ and by contracting each component of $D_G$ to a single vertex has positive surplus (as viewed from $A_G$);
\item [\rm{(v)}]$E_G(C_G,D_G)=\emptyset$.
\end{itemize}
\end{theorem}

The partition $(D_G,A_G,C_G)$ is called a \emph{canonical decomposition}. When there are no confusions, we also denote $G[D_G]$, $G[A_G]$ and $G[C_G]$ by $D_G, A_G$ and $C_G$, respectively.
For a maximum matching $M$ and a  component of $ D_i$ of $D_G$, we say that $D_i$ is \emph{M-full} if some vertex of $D_i$ is matched with a vertex in $A_G$, otherwise, $D_i$ is \emph{M-near full}.


Let $G$ be an $r$-regular graph without perfect matching.  
A connected component $D_i$  of $D_G$ is called ``\emph{good}" if $D_i$ is a non-trivial connected component  and satisfies the following two properties:
\begin{itemize}
  \item [(i)] $\alpha(D_i)=(|V(D_i)|-1)/2$;
  \item [(ii)] $D_i$ contains a maximum independent set $I(D_i)$ such that $E_G(I(D_i),A_G)=\emptyset$.
\end{itemize}

%
%



In this note, we character the extremal graphs for all $r$-regular graphs and obtain the following results.

\begin{theorem} Let $G$ be a connected $r$-regular graph.  Then $\alpha(G)=\mu(G)$ if and only if $G$ is bipartite or $(D_G,A_G,C_G)$ satisfies that
\begin{itemize}
\item [\rm{(i)}]$C_G=\emptyset$,
\item[\rm{(ii)}]$A_G\subseteq I(G)$ for any maximum independent set of $G$,
\item[\rm{(iii)}]every nontrivial component of $D_G$ is good.
\end{itemize}
\end{theorem}


\section {\large Proof of Theorem 6}

\noindent Before proving the Theorem 6, we firstly show the following lemma.

\begin{lemma}\label{main-lem} Let $G$ be a connected $r$-regular graph without perfect matching. If $\alpha(G)=\mu(G)$, then
\begin{itemize}
\item [\rm{(i)}]$A_G\subseteq I(G)$ for any maximum independent set of $G$;
\item [\rm{(ii)}]$C_G=\emptyset$.
\end{itemize}
\end{lemma}

\begin{proof}
Firstly, we show (i). Let $I(G)$ be an arbitrary maximum independent set of $G$, let $A_G'=I(G)\cap A_G$ and $B_G'=I(G)\cap B_G$, where $B_G\subseteq D_G$ denotes the set of isolated vertices of $D_G$. Let $q$ denote the number  of   connected components of $D_G$. Let $D_i$ denote  the    connected component  of $D_G$ for $1\leq i\leq q$.  By Theorem \ref{GE-thm} (iii), we have
\begin{align*}
\mu(G)&=\mu( C_G)+|A_G|+\mu(D_G)\\
&=\frac{1}{2}|C_G|+|A_G|+\frac{1}{2}\sum_{i=1}^{q}(|D_i|-1)£¬
\end{align*}
i.e.,
\begin{align}\label{muG}
\mu(G)=\frac{1}{2}|C_G|+|A_G|+\frac{1}{2}\sum_{i=1}^{q}(|D_i|-1)£¬
\end{align}
Since $D_i$ is factor-critical, we have $\alpha(D_i)\leq (|D_i|-1)/2$.
Thus we have
$$|I(G)\cap D_G|\leq \frac{1}{2}\sum_{i=1}^{q}(|D_i|-1).$$
By Theorem \ref{GE-thm} (ii), $C_G$ has a perfect matching. Thus we infer that
$$\alpha(C_G)\leq \frac{1}{2}|C_G|.$$
Hence,
\begin{align*}
\alpha(G)=|I(G)|&=|I(G)\cap C_G|+|I(G)\cap A_G|+|I(G)\cap D_G|\\
&\leq\alpha( C_G)+|A_G'|+|B_G'|+\alpha(D_G-B_G)\\
&\leq\frac{1}{2}|C_G|+|A_G'|+|B_G'|+\frac{1}{2}\sum_{i=1}^{q}(|D_i|-1),
\end{align*}
i.e.,
\begin{align}\label{inde(G)}
\alpha(G) \leq \frac{1}{2}|C_G|+|A_G'|+|B_G'|+\frac{1}{2}\sum_{i=1}^{q}(|D_i|-1),
\end{align}

\vspace{2mm}\noindent{\bf Claim 1.}~$B'(G)=\emptyset$.
\bigskip

By contradiction. Suppose that $B_G'\neq \emptyset$. Note that $\alpha(G)=\mu(G)$. Combining (\ref{muG}) and (\ref{inde(G)}), we have
\begin{align}\label{eq1}
|A_G|=|A_G'|+|B_G'|.
\end{align}
Since $G$ is an regular graph and $B_G'$ is an indepednet set, we have $|N_{G}(B_G')|\geq |B_G'|$ with equality if and only if
$G[N_{G}(B_G')\cup B_G']$ is a connected component of $G$ and $N_{G}(B_G')$ is also an independent set. Note that $G$ is connected. So if $|N_{G}(B_G')|= |B_G'|$, then $V(G)=B_G'\cup N_{G}(B_G')$ and $G$ is an $r$-regular bipartite graph, which implies that $G$ has a perfect matching by Hall's Theorem, a contradiction. Thus we may assume that $|N_{G}(B_G')|> |B_G'|$. Since $A_G'\cup B_G'$ is an independent set,  we have $A_G'\subseteq A_G-N_{G}(B_G')$. Thus
\[
|A_G'|+|B_G'|\leq |A_G-N_{G}(B_G')|+|B_G'|<|A_G|,
\]
contradicting to (\ref{eq1}). This completes the proof of claim 1.

%


\medskip

By Claim 1, $|A_G|=|A_G'|$, then we have $A_G=A_G'\subseteq I(G).$ This completes the proof of (i).

\vspace{2ex}

Next we show (ii).  Suppose that the result does not hold. Since $\alpha(G)=\mu(G)$,  by (\ref{muG}) and (\ref{inde(G)}), we have $$|I(G)\cap C_G|=\alpha( C_G)=\mu( C_G)=\frac{1}{2}|C_G|.$$
Recall that $A_G\subseteq I(G)$. One can see that $E_G(A_G, I(G)\cap C_G)=\emptyset$. Since $G$ is $r$-regular, we have
\[
\frac{1}{2}r|C_G|=r|I(G)\cap C_G|\leq e_G(I(G)\cap C_G, C_G-(I(G)\cap C_G))\leq r|C_G-(I(G)\cap C_G)|=\frac{1}{2}r|C_G|,
\]
which implies $E_G(A_G, C_G-(I(G)\cap C_G))=\emptyset$. Thus we have $E_G(A_G,C_G)=\emptyset$. Note that $E_G(D_G,C_G)=\emptyset$ by Theorem \ref{GE-thm} (v). On the other hand, since $G$ contains no perfect matchings, one can see that $D_G\neq\emptyset$ by definition of $D_G$.  Since $G$ is  connected, we may infer that $C_G=\emptyset$. This completes the proof of Lemma \ref{main-lem}.
%
%
%
%
\end{proof}





\begin{proof}[\textbf{Proof of the Theorem 6.}]~
Firstly, we consider sufficiency.
Let $G$ be an $r$-regular bipartite graph with bipartition $(A,B)$. One can see that $|A|=|B|$ and $\alpha(G)=\frac{|V(G)|}{2}.$ By Hall's  Theorem, $G$ has a perfect matching, i.e., $\mu(G)=\frac{|V(G)|}{2}.$ Therefore, $\alpha(G)=\mu(G)$.

Now we may assume that $G$ is a regular graph and satisfies the following three conditions
\begin{itemize}
  \item [(i)]$C_G=\emptyset$,
\item[(ii)]$A_G\subseteq I(G)$ for any maximum independent set of $G$,
\item[(iii)]every nontrivial component of $D_G$ is good.
\end{itemize}
Let $q$ denote the number of  connected components of $D_G$ and let $D_i$ denote the connected component of $D_G$ for $1\leq i\leq q$. By Theorem \ref{GE-thm} (iii),we have
\begin{align}\label{main-m-number}
\mu(G)&=|A_G|+\frac{1}{2}\sum_{i=1}^{q}(|D_i|-1).
\end{align}
Since $D_i$ is good for $1\leq i\leq q$, then $D_i$ contains an independent set $I(D_i)$ such that
\begin{align*}
|I(D_i)|=(|V(D_i)|-1)/2\ \mbox{and } E_G(I(D_i),A_G)=\emptyset.
\end{align*}
When $D_i$ is an isolated vertex, $I(D_i)=\emptyset$. So $I(G)=A_G\bigcup \cup_{i=1}^qI(D_i)$ is an independent set of $G$. Note that
\begin{align}\label{main-I-number}
|I(G)|=|A_G|+\frac{1}{2}\sum_{i=1}^q(|D_i|-1).
\end{align}
Since $I(G)$ is a  maximum independent set, combining (\ref{main-m-number}) and (\ref{main-I-number}), one can see that
\[
\mu(G)=|I(G)|= \alpha(G).
\]

%
%

\vspace{2ex}
Next, we prove the necessity. Let $G$ be an $r$-regular graph with $\alpha(G)=\mu(G)$. Let $I(G)$ be a maximum independent set of $G$. We discuss two cases.

\vspace{2mm}\textbf{Case 1.}~$G$ has a perfect matching.
\vspace{2mm}

Note that
$$\mu(G)=\frac{1}{2}|V(G)|=\alpha(G)$$
and
\begin{equation}
\begin{aligned}
|I(G)|=\frac{|V(G)|}{2}=|V(G)-I(G)|.
\end{aligned}
\end{equation}
One can see that
\begin{align*}
e_{G}(I(G),V(G)-I(G))=\alpha(G)r=\frac{|V(G)|}{2}r.
\end{align*}
It follows   that $V(G)-I(G)$ is an independent set and $G$ is an $r$-regular bipartite graph.

\vspace{2mm}\textbf{Case 2.}~$G$ has no perfect matching.
\vspace{2mm}

By Lemma \ref{main-lem},   $C_G=\emptyset$ and $A_G\subseteq I(G)$. Let $B_G$ denote the set of isolated vertices of $D_G$. Since $G$ is connected, then for every $x\in B_G$, $E_G(\{x\},A_G)\neq \emptyset$. So we have $B_G\cap I(G)=\emptyset$.
So it is sufficient for us to show that  every nontrivial component $D_i$ of $D_G$  is good.
 Since $D_i$ is factor-critical, we have $\alpha(D_i)\leq \frac{1}{2}(|D_i|-1)$. Recall that $A_G\subseteq I(G)$. Then we have
\[
\alpha(G)=|I(G)|=|A_G|+\sum_{i=1}^p |V(D_i)\cap I(G)|\leq |A_G|+\frac{1}{2}\sum_{i=1}^p(|D_i|-1),
\]
where $p$ denotes the number of connected components of $D_G$ with order at least three.
Note that
\[
\alpha(G)=\mu(G)=|A_G|+\frac{1}{2}\sum_{i=1}^p(|D_i|-1).
\]
Hence we have $|I(G)\cap V(D_i)|=\frac{1}{2}(|D_i|-1)$ and so $I(G)\cap V(D_i)$ is a maximum independent set of $D_i$. Moreover, one can see that
$E_G(I(G)\cap V(D_i), A_G)=\emptyset$ since $A_G\subseteq I(G)$. This completes the proof.
\end{proof}






\end{document}